\documentclass{amsart}
\usepackage[dvips]{color}
\usepackage{graphicx}

\usepackage{latexsym}
\usepackage{amssymb}
\usepackage{amsmath}
\usepackage{amsthm}
\usepackage{amscd}
\theoremstyle{plain}
\newtheorem{thm}{Theorem}
\newtheorem{lem}{Lemma}
\newtheorem{Bob}{Definition}
\newtheorem{prop}{Proposition}

\theoremstyle{remark}

\newtheorem{claim}{Claim}

\newtheorem{cor}{Corollary}
\title[Ergodic homogeneous algorithms]{Ergodic  homogeneous multidimensional  continued fraction algorithms}
\begin{document}

\maketitle

\centerline{ Jonathan Chaika \footnote{Partially supported by NSF grant DMS 1004372.}}
\smallskip
{\footnotesize
 \centerline{Departament of Mathematics, University of Chicago}
   \centerline{411 Eckhart Hall, 5734 S. University Avenue, Chicago, Illinois 60637, USA}
   \centerline{jonchaika@math.uchicago.edu}
}

\medskip

\centerline{ Arnaldo Nogueira\footnote{Partially supported by  ANR Perturbations.\\
{\it{Mathematical Subject Classification (2000)}: 11K55, 28D99.}}}
\smallskip
{\footnotesize
 \centerline{Institut de Math\'ematiques de Luminy,  Aix-Marseille Universit\'e}
 \centerline {163, avenue de Luminy - Case 907, 13288 Marseille Cedex 9, France}
 \centerline{arnaldo.nogueira@univ-amu.fr}} 

\noindent 
\today\\\

\begin{abstract}
\noindent
Homogeneous continued fraction algorithms are multidimensional generalizations of the classical Euclidean algorithm, 
the dissipative map
$$
(x_1,x_2) \in \mathbb{R}_+^2 \longmapsto
\left\{ \begin{array}{ll}         
         (x_1 - x_2, x_2), & \mbox{if $x_1 \geq x_2$} \\
         (x_1, x_2 - x_1), & \mbox{otherwise.}
\end{array}
\right.
$$
We focus on those which 
 act piecewise linearly on finitely many copies of positive cones which we call Rauzy induction type algorithms. 
 In particular, a variation  Selmer algorithm  belongs to this class. 
We prove that Rauzy induction type algorithms, as well as Selmer algorithms, are ergodic with respect to Lebesgue measure. 
\end{abstract}

\section{ Introduction}

\noindent
Here we study the ergodic properties of homogeneous continued fraction algorithms. 
We have two motivations. First to prove the ergodicity of multidimensional continued fraction algorithms, among them  Selmer algorithm. Second, to provide a proof of Veech's result \cite{metric} that (non-normalized) Rauzy induction algorithm acting on the space of interval exchange transformations 
is ergodic with respect to Lebesgue measure, without appealing to zippered rectangles and Teichm\"{u}ller theory, whose set up cannot be extended to other algorithms. In particular our proof works for any  Rauzy induction type algorithm.
We recall that our result  will imply that any normalization of all these algorithms is also ergodic. 

Considering our first motivation: these algorithms have been studied with the hopes of advancing classical results on diophantine approximation from dimension 1 to higher dimensions. 
 It is likely that the methods of this paper could be used to prove some diophantine results. Indeed, Chaika used similar methods to prove \cite[Theorem 9]{iet kurz}.
It seems that the ergodicity of these algorithms can be used to prove some multidimensional diophantine inhomogeneous approximation results as in Laurent and Nogueira \cite{lan} for the  one-dimensional case.

Among references about continued fraction algorithms, we mention  \cite{pu},  \cite{lag},  \cite{bgu} and  \cite{sc}.   Beyond diophantine approximation, we recall that the Jacobi-Perron algorithm was introduced  with the aim of obtaining an extension of Lagrange's theorem which characterizes quadratic surds using the classical continued fractions. These algorithms also appear naturally  in  situations which go beyond Dynamical Systems and Number Theory. For example, in \cite{km} an algorithm is used to describe a model in Percolation Theory and in \cite{bch}  it is introduced an  algorithm that yields all infimum sequences of finitely many  letters ordered lexicographically.

Considering  our second motivation: Veech's original motivation for proving the ergodicity of Rauzy induction was to show that a certain invariant of  interval exchange transformations (the Sah-Arnoux-Fathi invariant) was not measurable. He related the ergodicity of this infinite measure dissipative action to a finite measure preserving action on a different finite measure space.
 He used extensive results from \cite{gauss} and \cite{metric} to set up the Hopf argument and show that the action on the finite measure space was ergodic.
  He then showed that the Rauzy induction was ergodic. 
  Our approach can be thought of as an analogy with the Hopf argument where induction contracts the past and expands the future. 
  We use probability theory and results of Kerchoff \cite{ker} which rely on elementary methods in place of ergodic theory.
    It is likely that one could use the theory of random Markov compacta developed by Bufetov
     in \cite{buf} to extend Veech's original proof of ergodicity more directly to the settings we address.

The Rauzy induction algorithm acts on a directed graph whose nodes are permutations belonging to a Rauzy class. We call 
{\it Rauzy induction type algorithm} any algorithm which acts on a directed graph which has the same features of a Rauzy graph.  
Throughout the paper all these algorithms will be denoted  by $\mathcal{I}$. 
Our main result is the following.

\begin{thm} \label{main}
Let $\mathcal{I}:\mathbb{R}^d_+  \times \mathcal{P}\rightarrow \mathbb{R}^d_+  \times \mathcal{P}$ be a Rauzy induction type algorithm, 
then the map $\mathcal{I}$ is ergodic  with respect to Lebesgue measure.\end{thm}

 In fact once established that they are ergodic with respect to Lebesgue measure it  implies that they are exact, so they satisfy a $0-1$ law. A result proved by Miernowski and Nogueira  \cite{mie} says that a non-singular ergodic map $T$ is exact  if, and only if, it fulfills the following intersection property: for every positive measure set $A$ there exists a $k=k(A) \geq 0$ such that the set $T^k(A) \cap T^{k+1}(A)$ is a positive measure set. There it is proved that the Rauzy induction algorithm is exact. A Rauzy induction type algorithm keeps all properties of the Rauzy induction algorithm responsible for the intersection property, in particular the existence of a {\it loop} in the graph, 
 therefore  the same claim holds for all Rauzy induction type algorithms. 
 
 \begin{cor} \label{exact} Let $\mathcal{I}$ be a Rauzy induction type algorithm, 
then the map $\mathcal{I}$ is exact  with respect to Lebesgue measure.\end{cor}

Our approach can also be applied to Selmer algorithm (see \cite{sc}).
Let 
$$
\Sigma_d=\{\bar x=(x[1] ,\ldots,  x[d]) \in \mathbb{R}^d_+:x[1] \leq \ldots \leq  x[d]\}
$$ 
and $\sigma_{\bar x}$ be the permutation which arranges $x[1] ,\ldots, x[d]-x[1]$ in ascending order. 
The map 
\begin{equation}\label{selmer}
S:\bar x \in \Sigma_d \mapsto 
 \sigma_{\bar x} (x[1] ,\ldots,  x[(d-1)], x[d]-x[1]) \in \Sigma_d
 \end{equation}
is  called {\it Selmer algorithm}. As an application of Theorem \ref{main},  we obtain.

\begin{cor}\label{exact2} 
The Selmer  algorithm  is  ergodic   with respect to Lebesgue measure.
\end{cor}

The paper is organized as follows. In Section 2 we define interval exchange transformations and the Rauzy induction algorithm.  
In Section 3 we extend the definition of the Rauzy induction algorithm, where the Rauzy graph is replaced by any graph which fulfills  the hypotheses  in the end of Section 3. Our main theorem, Theorem \ref{main}, holds for any  Rauzy type algorithm. 
 The proof of our main theorem (Theorem \ref{main}) is in  Section 4.  Our proof is based on the work of Kerckhoff \cite{ker}.
Section 5 is devoted to the  Selmer algorithm. There we show that the dynamics of Selmer algorithm can be described by a variation of the Rauzy type algorithm. However one of the assumptions of a Rauzy type algorithm which concerns a central property of the simplicial system,  is not fulfilled. This property is proved in Proposition \ref{selmer balanced general}. 
For completeness, in Section 6 
we give examples of  classical homogeneous algorithms.  We also include examples of homogeneous continued fraction algorithms which  are not ergodic.

\section{Homogeneous  algorithms}

\noindent
Let $\mathbb{R}_+^d=\{ (x[1],\ldots, x[d])\in \mathbb{R}^d :  x[i] \geq 0,1\leq  i \leq d\}$ denote the positive cone, where $d\geq2$. Throughout the paper $\bar x$ will also be identified with the corresponding column-vector, so, if $M$ is a $d\times d$-matrix,   $M\bar x$  is well defined.

The {\it homogeneous Euclidean algorithm} is the  map defined by
\begin{equation}\label{E1}
\mathcal{I}: \bar x=(x[1] ,x[2]) \in \mathbb{R}_+^2 \longmapsto
\left\{ \begin{array}{ll}         
         (x[1] - x[2], x[2]), & \mbox{if $x[1] \geq x[2]$} \\
         (x[1], x[2] - x[1]), & \mbox{otherwise.}
\end{array}
\right.
\end{equation}
Let $x[1]$ and $x[2]$ be positive integers, thus the action of successive powers of $\mathcal{I}$ on  $\bar x=(x[1],x[2])$ corresponds to the application of the Euclidean algorithm for finding the \emph{greatest common divisor}  of $x[1]$ and $x[2]$,  say $\delta$. There exists a $k\geq 1$ such that  $\mathcal{I}^k(\bar x)=(0,\delta)$ which is the source of the name for the transformation $\mathcal{I}$. 

Let $\gamma_1=\left( \begin{array}{cc}
1 &   1 \\
0  & 1 
\end{array} \right)$ and $\gamma_2=\left( \begin{array}{cc}
1 &   0 \\
1  & 1 
\end{array} \right)$ be the elementary matrices of $SL(2,\mathbb{Z})$.
 In matricial form (\ref{E1}) is given by
$$
\mathcal{I}: \bar x=\left( \begin{array}{c}
x[1] \\
x[2]
\end{array} \right) \in \mathbb{R}_+^2 \longmapsto
\left\{ \begin{array}{ll}         
         \gamma_1^{-1}\bar x, & \mbox{if $x[1] \geq x[2]$} \\
         \gamma_2^{-1}\bar x, & \mbox{otherwise}
\end{array}
\right.
$$ 
and has a striking property (see \cite{poin}): for every $\bar x \in \mathbb{R}_+^2$, the full orbit of $\bar x$
$$
\cup_{k\geq 0}\cup_{l\geq 0} \mathcal{I}^{-l}(\{\mathcal{I}^{k}(\bar x)\})= SL(2,\mathbb{Z}) \{ \bar x\} \cap  \mathbb{R}_+^2.
$$

We call {\it homogeneous algorithms} maps which are multidimensional generalizations  of (\ref{E1}).

Next we present the Rauzy induction algorithm acting on the space of interval exchange transformations which was introduced by Rauzy (see \cite{ra}). For completeness sake we begin with the definition of interval exchange transformation.

\subsection{Interval exchange transformation}

Here our main reference is Veech \cite{gauss}. Let  $\mathfrak{S}_d$ be the set of permutations on $d$ letters.  An exchange of $d$ intervals is defined by two parameters $(\bar x,\pi)\in \mathbb{R}_+^d\times \mathfrak{S}_d$.  Let $I^{\bar x}=[0,\vert \bar x \vert )$, where 
$\bar x=( x[1] ,\ldots, x[d])$. We set $\alpha_0(\bar x)=0$ and $\alpha_i(\bar x)=x[1] +\ldots+ x[i]$, for $1\leq i \leq d$. The points $\alpha_i(\bar x)$ partition the interval $I^{\bar x}$ into $d$ subintervals $I^{\bar x}_i=[\alpha_{i-1}(\bar x),\alpha_i(\bar x) )$ of length $\bar x_i$ and $\pi$ is used to permute the subintervals $I^{\bar x}_i$. We set $\bar x^{\pi}=( x^{\pi}[1], \ldots ,  x^{\pi}[d])$, where  $x^{\pi}[i] = x[\pi^{-1}(i)]$, for $1\leq i \leq d$. 
The one-one onto map  defined by 
$$
y \in  I^{\bar x}_i \mapsto y- \alpha_{i-1}(\bar x) + \alpha_{\pi(i)-1}(\bar x^{\pi}) , \; \mbox{ for } \; 1\leq i \leq d,
$$
is the so called $(\bar x,\pi)$-{\it interval exchange} $T_{(\bar x,\pi)}:I^{\bar x}\to I^{\bar x^{\pi}}_{\pi(i)}= I^{\bar x}$.  
The map $T_{(\bar x,\pi)}$ acts as a translation on each subinterval $I^{\bar x}_i$, thus it preserves the Lebesgue measure. 

We say  that a permutation $\pi \in \mathfrak{S}_d$ is $ irreducible$, if $1\leq k\leq d$ and $\pi\{1,\ldots,k\}= \{1,\ldots,k\}$ imply $k=d$. In others words, for an irreducible permutation $\pi$,  if $y>0$ and $T_{(\bar x,\pi)}([0,y))=[0,y)$, then $y=\alpha_d(\bar x)$. We denote by $\mathfrak{S}_d^0$ the set of irreducible permutations of $\mathfrak{S}_d$.

If $\pi$ is not irreducible, for every $\bar x\in\mathbb{R}^d_+$ the corresponding $(\bar x,\pi)$-interval exchange may be seen as two separated exchanges of $k$ and $d-k$ intervals. In particular it is not ergodic with respect to Lebesgue measure. In what follows only irreducible permutations will be considered.

\subsection{Rauzy inductive process}

Here we follow  \cite[Section 2]{nru}. Let $T_{(\bar x,\pi)}$ be an interval exchange given by an irrational vector $\bar x$ which means that its coordinates are rationally independent and $\pi \in \mathfrak{S}_d^0$. The so-called {\it Rauzy induction} assigns to $T_{(\bar x,\pi)}$ a first return map induced on a suitable subinterval of $I^{\bar x}$. We partition $\mathbb{R}_+^d$ into two sub-cones   
$$
\mathcal{C}'=\{\bar x :  x[d]> x^{\pi}[d]\} \; \mbox{ and } \;\mathcal{C}''=\{\bar x : x^{\pi}[d] >  x[d] \}
$$
and define the induction on each of them separately. If $\bar x \in \mathcal{C}'$, we define
$$
T':[0, \alpha_{d-1}(\bar x^{\pi})) \rightarrow [0, \alpha_{d-1}(\bar x^{\pi}))
$$
to be the first return map induced by $T_{(\bar x,\pi)}$ on the interval $[0, \alpha_{d-1}(\bar x^{\pi}))$. A computation shows that $T'$ is still a $d$-interval exchange. The couple of parameters $(\bar x',\pi')$ corresponding to $T'$ is described as follows. Consider the $n\times n$-matrix
$$ 
A_{\pi}'=\left( \begin{array}{ccccccc}
1 &   & & & & & \\
   & 1& & & & 0 & \\
   &    & & \ddots & && \\
   & 0 & &&&& \\
   &&&-1&&&1
\end{array} \right),
$$
where $(A_{\pi}')_{d, \pi^{-1}d}=-1$. Then $\bar x'=A_{\pi}'\bar x$ and the permutation $\pi'$ is given by
$$
\pi'(j)= 
\left\{ \begin{array}{ll}         
         \pi(j), & \mbox{if $\pi(j)\leq  \pi(d)$,} \\
         \pi(j)+1, & \mbox{if $\pi(d)<\pi(j)< d$,}\\
         \pi(n)+1, & \mbox{if $\pi(j)=d$.}
         \end{array}
\right.
$$

If $\bar x \in \mathcal{C}''$, we define 
$$
T'':[0, \alpha_{d-1}(\bar x)) \rightarrow [0, \alpha_{d-1}(\bar x))
$$
by inducing $T_{(\bar x,\pi)}$ on the interval $[0, \alpha_{d-1}(\bar x))$. Then $T''$ is also an $n$-interval exchange. We consider the $GL(d,\mathbb{Z})$ matrix
\begin{equation}\label{pi1}
A_{\pi}''=\left( \begin{array}{ccccccc}
1 &             &    & & & & \\
   & \ddots &    &  & && \\
   &              & 1&   &            & &   -1\\
   &              &0 & 0&\ldots  &0& 1\\
   &              &   & 1&            && \\
    &             &   &   & \ddots & & \\
    &             &   &   &             & 1& 0
\end{array} \right),
\end{equation}
where $(A_{\pi}'')_{\pi^{-1}d,d}=-1$, and set $\bar x''=A_{\pi}''\bar x$. Let the permutation $\pi''$ be given by
\begin{equation}\label{pi2}
\pi''(j)= 
\left\{ \begin{array}{ll}         
         \pi(j), & \mbox{if $j \leq  \pi^{-1}(d)$} \\
         \pi(n), & \mbox{if $j= \pi^{-1}(d)+1$}\\
         \pi(j-1), & \mbox{otherwise.}
         \end{array}
\right.
\end{equation}
We have $T''=T_{(\bar x'',\pi'')}$.

Let $\pi_0\in \mathfrak{S}_d^0$ be a fixed permutation and define $\mathcal{P}$ to be the set of all permutations $\pi\in \mathfrak{S}_d^0$ which can be reached by the successive iterations of the Rauzy induction starting at some $T_{(\bar x,\pi_0)}$, $\bar x\in\mathbb{R}_+^d$.  The set $\mathcal{P}$ is called the {\it Rauzy class of permutations of $\pi_0$}, or the {\it Rauzy class of} $\pi_0$ for short.  

In order to study the possible sequences of permutations arising from this process,  we construct a directed graph $\mathcal{G}$ whose nodes are the permutations $\pi\in \mathcal{P}$. For every $\pi\in  \mathcal{P}$ an arrow goes from $\pi$ to each of $\pi'$ and $\pi''$ given by (\ref{pi1}) and (\ref{pi2}) respectively. For $n=2$ we have only one Rauzy class whose graph consists of one node with two loops attached. 

The next lemma allows us to iterate the above inductive process.
 
\begin{lem}[\cite{nru}, Lemma 2.4] 
Let $\bar x$ be irrational and $\pi$ irreducible. Then both $\bar x',\bar x''$ are irrational and both $\pi',\pi''$ irreducible. 
\end{lem}

Let $\pi$ be irreducible, so we associate to almost every pair $(\bar x,\pi)$ an infinite sequence $(\bar x^{(k)},\pi^{(k)})$ of points and an infinite  sequence  of elementary matrices. Denote by $A^{(k)}$ the product of the first $k$ elementary matrices so that $A^{(k+1)}=A^{(k)}E$, where $E$ is an elementary matrix. The point $\bar x^{(k)}$ belongs to the positive cone $A^{(k)}\mathbb{R}_+^d$. If $v_1,\ldots,v_d$ are the column-vectors of the matrix $A^{(k)}$, then there exist $1\leq j\neq k \leq d$ such that the column-vectors of the matrix $A^{(k+1)}$ equal $v'_1,\ldots,v'_d$ where $v'_i=v_i$, for every $i\neq j$, and $v'_j=v_j+v_k$. 

In order to prove that the interval exchange $T_{(\bar x,\pi)}$ is uniquely ergodic it suffices to establish that
$\displaystyle \cap_{k\geq 0} A^{(k)} \mathbb{R}_+^d$ is a one-dimensional set.

It has been proved in \cite{mie} that 

\begin{lem}  Let $\pi \in \mathfrak{S}_d^0$ be such that $\pi(d-1)=d$ and $\pi''$ be the permutation defined by (\ref{pi2}). Then $\pi''=\pi$. \end{lem}

The above lemma proves that, at such node, the Rauzy graph has a loop. 
\begin{Bob}\label{loop}
A {\it loop permutation}  is an irreducible permutation $\pi$ with either $\pi'=\pi$ or $\pi''=\pi$. 
\end{Bob}
Every Rauzy class has a loop permutation.

The following lemma concerns the structure of the Rauzy graph for $d\geq 3$.

\begin{lem}[\cite{nru}, Lemma 2.2 and 2.4] Let $\pi_0 \in \mathfrak{S}_d^0$ and $\mathcal{G}$ be its graph. For every $\pi_1, \pi_2 \in \mathcal{G}$ there is a path in $\mathcal{G}$ starting at $\pi_1$ and reaching $\pi_2$. Moreover, every $\pi\in  \mathcal{G}$ has exactly two followers and two predecessors in $ \mathcal{G}$. \end{lem}

\subsection{Rauzy induction algorithm}

Let $ \mathcal{P}$ be a Rauzy class in $\mathfrak{S}_d^0$. The inductive process described in the previous subsection defines an algorithm $ \mathcal{I}$ acting on the parameter space $\mathbb{R}_+^d\times  \mathcal{P}$ by 
\begin{equation}\label{Ind}
\mathcal{I}:  (\bar x,\pi) \in \mathbb{R}_+^d \times \mathcal{P}  \longmapsto \left\{ \begin{array}{lll} (\bar x',\pi')  & \text{if} & \bar x_n > \bar x_n^{\pi}, \\ (\bar x'',\pi'') & \text{if} & \bar x_n <\bar x_n^{\pi} 
\end{array} \right.
\end{equation} 
which is called the {\it Rauzy induction} of interval exchange transformations. 

The space $\mathbb{R}^d_+ \times \mathcal{P}$ is endowed with Lebesgue measure.

\begin{thm}[\cite{metric}, Theorem 1.6] 
For every Rauzy class $\mathcal{P}$, the map $\mathcal{I}$ is ergodic on $\mathbb{R}^d_+  \times \mathcal{P}$ with respect to Lebesgue measure.\end{thm}

In order to illustrate the definition of $ \mathcal{I}$, we will now describe explicitly its action in the easiest cases $d=2,3$. In what follows the permutations are represented in the form $\pi=(\pi^{-1}(1) \; \ldots \; \pi^{-1}(d))$.
\\

\noindent {\bf 1.} In the case of $d=2$, there is only one irreducible permutation on two letters, the transposition $(21)$ which constitutes its own Rauzy class. The action of $\ \mathcal{I}$ on the second coordinate is thus trivial. On the first coordinate $ \mathcal{I}$ acts as the Euclidean algorithm  defined by (\ref{E1}).

\noindent {\bf 2.} For $d=3$ there is also only one Rauzy class and it contains all irreducible permutations on three letters, namely $(231)$, $(321)$, $(312)$. The Rauzy induction is decribed as follows:
$$
\mathcal{I}(\bar x,(231))=
\left\{ \begin{array}{ll}         
         ((x[1],x[2],x[3]-x[1]), (231)), & \mbox{if $x[3] > x[1]$}, \\
         ((x[1]-x[3],x[3],x[2]), (321)), & \mbox{if $x[1] > x[3]$},
\end{array}
\right.
$$ 
$$
\mathcal{I}(\bar x,(321))=
\left\{ \begin{array}{ll}         
          ((x[1],x[2],x[3]-x[1]), (312)), & \mbox{if $x[3] > x[1]$}, \\
         ((x[1]-x[3],x[3],x[2]), (231)), & \mbox{if $x[1] > x[3]$},
\end{array}
\right.
$$
$$
\mathcal{I}(\bar x,(312))=
\left\{ \begin{array}{ll}         
                  ((x[1],x[2],x[3]-x[2]), (321)), & \mbox{if $x[3] > x[2]$}, \\
         ((x[1],x[2]-x[3],x[3]), (312)), & \mbox{if $x[2] > x[3]$},
\end{array}
\right.
$$
From the above expressions the Rauzy graph can be deduced.\\

\noindent 
{\bf 3.} One may check that for $d=4$ we get two distinct Rauzy classes of irreducible permutations, one generated by $(4321)$ and another by $(3412)$.

\section{Rauzy induction type algorithm}

Motivated by the Rauzy induction algorithm (\ref{Ind}),  we can define other homogeneous continued fraction algorithms which fit the approach proposed by Kerckhoff in \cite{ker}. Our first definition  
consists of considering a graph which is not  a Rauzy graph,  but it satisfies the same conditions.  Another definition will be illustrated in Section $5$ with the example of the Selmer  algorithm which is not defined in  the whole positive cone.

\subsection{Others graphs} 
Let $d\geq 2$ be fixed. Let  $\Omega$ be a finite set of symbols and $\mathcal{P}$ a subset of the set
$$
\{(i,j,\omega) \mid 1 \leq i<j\leq d,\omega \in \Omega \}.
$$
Next we consider a directed graph $\mathcal{G}$ whose nodes are the elements $\pi \in \mathcal{P}$ 
and associates to each node $\pi \in \mathcal{P}$ a pair of nodes $(\pi',\pi'') \in \mathcal{P} \times \mathcal{P}$, where $\pi'\neq \pi''$. It means that at every node $\pi$ there is a arrow leaving from $\pi$ to $\pi'$ and another one  from $\pi$ to $\pi''$.

The graph $\mathcal{G}$ allows us to define a homogeneous algorithm 
$$
\mathcal{I}: \mathbb{R}_+^d \times \mathcal{P} \rightarrow \mathbb{R}_+^d \times \mathcal{P}. 
$$
Let $\bar x \in \mathbb{R}_+^d$ and $\pi=(i,j, \omega) \in \mathcal{P}$, the map $\mathcal{I}$ is given by
\begin{equation}\label{type}
\mathcal{I}(\bar x,\pi)=
\left\{ \begin{array}{ll}         
         (( x[1], \ldots,x[i]- x[j], \ldots,x[j],\ldots,x[d]),\pi'), & \mbox{if $x[i] \geq x[j]$} \\
         (( x[1], \ldots,x[i], \ldots,x[j]- x[i],\ldots,x[d] ),\pi''), & \mbox{otherwise.}
\end{array}
\right.
\end{equation}
We remark that, for each of the above cases, there is an elementary  $d\times d$-matrix, denoted by $M((\bar x, \pi),1)$, such that, if 
$\mathcal{I}(\bar x,\pi)= (\bar x',\sigma)$, then
$$
M((\bar x, \pi),1) \bar x'= \bar x.
$$

\subsection{Example} Here we present  a $d$-dimensional Rauzy induction type algorithm which is not given by a Rauzy induction, for every $d\geq 3$.   The algorithm is defined by the  graph whose node set equals $\mathcal{P}=\{1,2,\ldots , d\}$. The node $\pi=i$ is associated to the pair of nodes $(\pi',\pi'')$, where $\pi'=i$ and 
$
\pi'' =\left\{ \begin{array}{ll}         
         i+1, & \mbox{if $i\leq d-1$} \\
         1, & \mbox{if $i=d$.}
\end{array}
\right.$

Let $\bar x \in \mathbb{R}_+^d$, we set for    $1\leq i \leq d-1$ 
$$
\mathcal{I}(\bar x,i) =
\left\{ \begin{array}{ll}         
         (( x[1], \ldots,x[i]- x[i+1], x[i+1],\ldots,x[d]),i), & \mbox{if $x[i] \geq x[i+1]$} \\
         (( x[1], \ldots,x[i], x[i+1]- x[i],\ldots,x[d]),i+1), & \mbox{if  $x[i+1] > x[i]$,} 
\end{array}
\right.
$$
otherwise
$$
\mathcal{I}(\bar x,d) =
\left\{ \begin{array}{ll}         
         (( x[1],  \ldots,x[d-1],x[d]- x[1]),d), & \mbox{if  $x[d] \geq x[1]$} \\
         (( x[1]- x[d], x[2], \ldots,x[d] ),1), & \mbox{if  and $x[1] > x[d]$.}
\end{array}
\right.
$$
We notice that $\mathcal{I}$ has a loop at each node. 

\subsection{Assumptions} 

We call {\it Rauzy induction type algorithm} any map $\mathcal{I}$ (\ref{type}) which satisfies the following$:$\\
1) The graph $\mathcal{G}$ contains no nontrivial isolated set in the meaning of \cite[p. 262]{ker}.\\ 
2) Each node of the graph has two incoming arrows which  come from distinct vertices.\\
3) The graph $\mathcal{G}$ has a {\it loop}, that is, there exists $\pi  \in \mathcal{P}$  such that $\pi\in \{ \pi',\pi'' \}$, so that the matrix norm grows polynomially.\\
4) Let $(\bar x,\pi) \in  \mathbb{R}_+^d \times \mathcal{P}$, where $\bar x$ is irrational, and $\mathcal{I}^k(\bar x,\pi)= (\bar{x}^{(k)},\pi^{(k)})$, for every $k\geq 0$, then $\bar x^{(k)}$ converges to the origin, as $k \rightarrow \infty$.\\

\medskip

In particular all above assumptions hold for the Rauzy induction algorithm (\ref{Ind}).

Under these assumptions, applying the approach given in \cite{ker} to study interval exchange transformations, we obtain  the ergodicity of the normalization of any Rauzy induction type algorithm. Precisely,  let 
$$
\Delta_{d-1}=\{(x[1] ,\ldots, x[d]) \in \mathbb{R}_+^d: \vert \bar x\vert=x[1 ]+\ldots + x[d]=1\}
$$ 
be the $(d-1)$-dimensional simplex, then the  normalization of $\mathcal{I}$ 
\begin{equation}\label{norm}
I:(\bar x,\pi) \in \Delta_{d-1} \times \mathcal{P} \mapsto (\frac{\bar x'}{\vert \bar x' \vert},\sigma) \in \Delta_{d-1} \times \mathcal{P}, \; \mbox{ where } \; \mathcal{I}(\bar x,\pi)= (\bar x',\sigma),
\end{equation}
is ergodic with respect to Lebesgue measure in $\Delta_{d-1} \times \mathcal{P}$.

The approach developed in  \cite{mie} can be applied to any Rauzy induction type algorithm. Theorem \ref{main} and  \cite{mie} together give  Corollary \ref{exact}.

\section{Proof of Theorem 1}

Next we recall  the notation that will be used.

Let $\mathcal{I}:\mathbb{R}_+^d \times \mathcal{P}\to \mathbb{R}_+^d\times \mathcal{P}$ be
the Rauzy induction type algorithm we want to show is ergodic, where  $\mathcal{P}$ is a finite set.

Let $\lambda$ denote Lebesgue measure on $\mathbb{R}_+^d\times \mathcal{P}$ and  $\lambda_0$ denote the normalized Lebesgue measure on $\Delta \times \mathcal{P}$.
In an abuse of notation $\lambda$ will also denote Lebesgue measure on $\mathbb{R}_+^d$ and $\lambda_0$ will also denote Lebesgue measure on $\Delta$.

Throughout of the section let $\bar x$ be understood to be a vector in $\mathbb{R}_+^d$ or
$\Delta$ as appropriate.

Let $\Delta_{a,b}=\{\bar x =(x[1],\ldots ,x[d])  \in \mathbb{R}_+^d \mid  |\bar x|=\sum x[i] \in [a,b]\}$. So, $\Delta=\Delta_{1,1}$.

Let $p_{\Delta}:\mathbb{R}^d_+ \to \Delta$ by $\displaystyle p_{\Delta}(\bar v)=\frac{\bar v}{|\bar v|}$.

Let $\pi_1: \mathbb{R}^d_+ \times \mathcal{P}\to \mathbb{R}^d_+$ and $\pi_2: \mathbb{R}^d_+ \times \mathcal{P}\to \mathcal{P}$ be projections on the first and second coordinate, respectively.

Let $M((\bar x,\sigma), n)$ be the matrix that goes with $n$ steps of induction for $(\bar x,\sigma)$. When there is no confusion $\sigma$ is suppressed. This matrix is chosen so that 
$$
M(\bar x,n) \mathcal{I}^n\bar x=\bar x.
$$

If $M$ is a matrix let $M_{\Delta}$ denote the set $M\mathbb{R}_+^d \cap \Delta$.

Let $I: \Delta\times \mathcal{P} \to \Delta\times \mathcal{P}$ be the normalization defined  by  (\ref{norm}).

Let  $M=(M[{i,j}])$ be a real positive $(d\times d)$-matrix, we define $\displaystyle C_j(M)= \sum_{1\leq i \leq d} M[{i,j}]$, for every $1\leq j \leq d$, and   $\displaystyle C_{max}(M)= max\{C_j(M) :1\leq j \leq d\}$.

Let $\beta >1$,  a positive $d\times d$-matrix $M$ is  $\beta$-$balanced$, if
$$
 \beta C_j(M) \geq C_{max}(M).
$$
We denote by $\mathcal{M}_{\beta}$ the set of all $\beta$-balanced matrices.

We show Theorem \ref{main} by showing that if $V \subset \mathbb{R}_+^d\times \mathcal{P}$ with $\lambda(V)>0$, $\mathcal{I}(V)\subset V$ and 
$\mathcal{I}^{-1}(V)\subset V$, then  $\lambda(V^c)=0$.

The following two propositions help establish Theorem \ref{main}.
Let $S' \subset \Delta_{1,2}$ with $\lambda(S')>0$.
Let $S= S' \times \mathcal{P}$.

\begin{prop}\label{hits well} Let $U\subset \Delta$ be a nonempty open set. Then there exist $C_{\mathcal{I}}>0$ which depends only on $\mathcal{I}$ and $k_0 \in \mathbb{N}$  which depends only on   $U$ such that for any $k>k_0$ 
$$\lambda\left(  \left(\Delta_{\beta^{-1}2^k,2^{k+1}} \cap p^{-1}_{\Delta}(U)\right)\times \mathcal{P} \cap
 \underset{i=1}{\overset{\infty}{\cup}}\mathcal{I}^{-i}(S)\right)> C_{\mathcal{I}}\lambda_0(U)\lambda(\Delta_{\beta^{-1}2^k,2^{k+1}})\frac{\lambda(S)}{\lambda(\Delta_{1,2})}.$$
\end{prop}

Informally this proposition states that $ \underset{i=1}{\overset{\infty}{\cup}}\mathcal{I}^{-i}(S)$ intersects wedges (that
is $p^{-1}_{\Delta}(U) \times \mathcal{P}$) in $\Delta_{\beta^{-1} 2^{k},2^{k+1}}$
about as much as one would expect. The following Corollary is a straightforward consequence of the ergodicity of $I$.

\begin{cor} \label{hits well perm} There exist $q$ depending only on  $\lambda_0(U)$ and $k_0$ depending only on   $U$ such that for any $k>k_0$ and each $\sigma \in  \mathcal{P}$ we have that
 $$\lambda\left( \left(\Delta_{\beta^{-1}2^k,q2^{k+1}} \cap p^{-1}_{\Delta}(U)\right)\times \{\sigma\}\cap\underset{i=1}{\overset{\infty}{\cup}} \mathcal{I}^{-i}(S)\right)
 > C^{(2)}_{\mathcal{I}}\lambda_0(U)\lambda(\Delta_{\beta^{-1}2^k,2^{k+1}})\frac{\lambda(S)}{\lambda(\Delta_{1,2})}.$$
\end{cor}

\begin{cor} \label{hits real well}There exists $q'$ depending only on  $\lambda_0(U)$ and $k_0$ depending only on $U$ such that for any $k>k_0$ and each pair $\sigma, \sigma' \in  \mathcal{P}$ we have that
 \begin{multline}\lambda\left( \left(\Delta_{\beta^{-1}2^k,q'2^{k+1}} \cap p^{-1}_{\Delta}(U)\right)
 \times \{\sigma\}\cap \underset{i=1}{\overset{\infty}{\cup}}\mathcal{I}^{-i}(S'\times \{\sigma'\})\right)
 \\> C^{(3)}_{\mathcal{I}}\lambda_0(U)\lambda(\Delta_{\beta^{-1}2^k,2^{k+1}})\frac{\lambda(S)}{\lambda(\Delta_{1,2})}.
 \end{multline}
\end{cor}

Note the $k_0$ may be different than in Proposition \ref{hits well}.

\begin{prop}\label{compare} If $U \subset \Delta$ is measurable with $\lambda(U)>0$, $\sigma, \sigma' \in \mathcal{P}$, $q'>1$ is as in Corollary \ref{hits real well}, $\epsilon>0$ and $r,l \leq \frac{q'}{\epsilon}$, then there exist $k_0, c>0$ depending on $q'$  such that for all $k>k_0$
$$
 \frac{ \lambda\left(\left(p_{\Delta}^{-1}(U)
\cap \Delta_{2^k(1+r\epsilon), 2^{k}(1+(r+1)\epsilon)}\right)\times \{\sigma\} \cap
\underset{i=1}{\overset{\infty}{\cup}}\mathcal{I}^{-i}(S)\right)}{\lambda\left(\left(p_{\Delta}^{-1}(U) \cap \Delta_{2^k(1+l\epsilon),
2^{k}(1+(l+1)\epsilon)}\right) \times \{\sigma\} \cap \underset{i=1}{\overset{\infty}{\cup}}\mathcal{I}^{-i}(S' \times \{\sigma'\})\right)}>c.
$$
\end{prop}

Informally, this proposition states that locally $\underset{i=1}{\overset{\infty}{\cup}}\mathcal{I}^{-i}(S)$ intersects
slices of wedges with comparable volume.

The proof of Theorem \ref{main} also requires another result (Lemma \ref{run condition}) that allows one to apply the propositions at any scale which will be stated later.\\

\subsection{ Induction dynamics}
Here we make more precise the assumptions which were made in the end of Subsection 3.3.

 \begin{enumerate}
\item It contains a loop permutation (see Definition \ref{loop}). This is stronger than what we need. We need that there is a sequence of steps of induction which start and end at the same permutation such that the associated matrix has powers whose operator norm grows polynomially.
 \item For Lebesgue almost every $(\bar x,\sigma) \in \mathbb{R}^d_+ \times \mathcal{S}$ we have that 
 $$
 \underset{n \to \infty}{\lim} \pi_1(\mathcal{I}^n(\bar x,\sigma))=\bar 0.
$$
\item \label{balance assumption}
There exist $\rho>0,K>1$ and $\nu_0>1$ depending only on $\mathcal{I}$ such that for any matrix of induction $M'=M((\bar x,\sigma),n)$ we have
\begin{multline*}
\lefteqn{\lambda_0(\{T: \pi(T)=\sigma, T \in M'_{\Delta} \exists m>n \text{ such that } M(T,m) \text{ is }}\\ \nu_0 \text{-balanced and }  C_{max} (M(T,m)) <K^d C_{max}(M')\})> \rho \lambda_0 ( M'_{\Delta}).\end{multline*}
This result is \cite[Corollary 1.7]{ker} which is a consequence of Assumption 1 in Subsection 3.3.
\item The graph of $\mathcal{P}$ is connected.
\end{enumerate}

 First let us notice that
$$
\underset{n=1}{\overset{\infty}{\cup}}\mathcal{I}^{-n}(B(\bar x,\delta_1)\times \{\sigma\})= \underset{n, (\bar y,\sigma'):\pi_2(\mathcal{I}^n(\bar y,\sigma'))=\sigma}{\cup}M((\bar y,\sigma'),n)B(\bar x,
\delta_1)\times \{\sigma'\}.
$$

\subsection{Proof of Proposition \ref{hits well}}
For readability reasons we suppress the permutation. Permutations will be introduced in the the proofs of Corollaries \ref{hits well perm} and \ref{hits real well}.

We first prove an independence result of induction and then use it
to prove the proposition.

\begin{lem} \label{I condition} There exists a constant $c_{\mathcal{I}}$ such that
 for any $\bar v \in \Delta$, $n\in \mathbb{N}$ there exist $k_0$ depending on $\bar v$ and $n$ so that if 
 $k>k_0$ 
\begin{multline}\lefteqn{\lambda_0(\{\bar z \in
M(\bar v ,n)_{\Delta}: \exists q \text{ with } M(\bar z,q)\in
\mathcal{M}_{\beta}, C_{max}(M(\bar z, q))\, \in
[2^k,2^{k+1}]\})}\\ \geq c_{\mathcal{I}} \lambda_0(M(\bar v,n)_{\Delta}).\end{multline}
\end{lem}

\begin{proof}[Proof of Lemma \ref{I condition}] 
Consider $k>k_0$ and we verify the lemma for this $k$. Let $\bar y \in M(\bar x, n)_{\Delta}$ and $K$ be as in Assumption 3
$$
m_{\bar{y}}=\max\{m: C_{max}(M(\bar{y},m)) <2^{k+1}K^d\}.
$$

Let $\beta=K^d\nu_0$ where $\nu_0$ is as in Assumption 3.
 Let $\bar{z}\in M(\bar{y},m_{\bar{y}})_{\Delta}$ be an element of the measure at least
$\rho\lambda_0(M(\bar{y},m_{\bar{y}})_{\Delta})$ set 
contained in $M(\bar{y},m_{\bar{y}})_{\Delta}$ guaranteed by Assumption 3. 
So there exists $q_{\bar{z}}\geq m_{\bar{y}}$ with 
$$
M(\bar{z},q_{\bar{z}})\in \mathcal{M}_{\nu_0} \text{ and } C_{max}(M(\bar{z},q_{\bar{z}}) \in [\frac{2^{k+1}}{K^d},2^{k+1}].
$$
Because $C_{\max}(M(\bar{u},r+1)) \leq 2 C_{\max}(M(\bar{u},r))$ for any $\bar{u},r$,
there exists $q'_{\bar{z}}\geq q_{\bar{z}}$ such that $C_{\max}M(\bar{z},q'_{\bar{z}}) \in [2^k,2^{k+1}]$. 
Because $C_{\max}M(\bar{z},q'_{\bar{z}}) \leq K^d C_{\max}M(\bar{z},q_{\bar{z}})$ it follows that $M(\bar{z},q'_{\bar{z}})\in\mathcal{M}_{\nu_0K^d}=\mathcal{M}_{\beta}. $ Treating the various $M(\bar{y},m_{\bar{y}})_{\Delta}$ which form a partition of a full measure subset of $ M(\bar{v},n)_{\Delta}$ the lemma follows.
\end{proof}

\begin{prop} \label{bal} (Kerckhoff \cite[Corollary 1.2]{ker}) If $M$ is $\beta$-balanced and
$W \subset \Delta$ is a measurable set, then
$$
\frac{\lambda_0(W)}{\lambda_0(\Delta)}< \frac{\lambda_0(p_{\Delta}(MW))}{\lambda_0(M_{\Delta})}\beta^{d}.
$$ 
\end{prop}

\begin{lem} \label{matrices} There exists a constant $c'_{\mathcal{I}}$ such that for any  open set $U \subset \Delta$ and all large enough $k$ (depending on $U$) there exists a set
 $\mathcal{V}_k(U) \subset \mathcal{M}_{\beta}$
  with $|\mathcal{V}_k(U)|\geq c'_{\mathcal{I}}\lambda_0(U)2^{kd}$
   such that for any $M \neq N\in \mathcal{V}_k(U)$ the following holds
    \begin{enumerate}
\item $C_{\max}(M) \in [2^k,2^{k+1}]$
\item $M_{\Delta}\subset U$
\item $N_{\Delta} \cap M_{\Delta}=\emptyset.$
\end{enumerate}
\end{lem}
\begin{proof} Because every open set in $\Delta$ is a disjoint union of $M_{\Delta}$ up to a null measure set it suffices to consider $U=M(\bar v,n)_{\Delta}$. Consider $M(\bar{z},q'_{\bar{z}})$ given by the proof of Lemma \ref{I condition}.  Let these matrices be ordered by $M \leq N$ if $M_{\Delta}\subset N_{\Delta}$. Consider the set of maximal elements in this set $\mathcal{V}_k(U)$. Lemma \ref{I condition} provides a set $\mathcal{V}_k(U) \subset \mathcal{M}_{\beta}$ where Conclusions 1 and 2 hold. The fact that steps of induction define partitions (so if $\bar{w} \in M(\bar{u}_1,n)_{\Delta}\cap M(\bar{u}_2,n)_{\Delta}$ then $M(\bar{u}_1,n)_{\Delta}=M(\bar{u}_2,n)_{\Delta}$)  and $\mathcal{V}_k(U)$ consist of maximal elements provides Conclusion 3. Each 
$$
C_i(M(\bar{z},q'_{\bar{z}})) \geq C_{\max}(M(\bar{z},q'_{\bar{z}})) \beta^{-1}\geq 2^k \beta^{-1}.
$$ 
So by Proposition \ref{bal} it follows that $\lambda_0(M(\bar{z},q'_{\bar{z}})_{\Delta})\leq 2^{-kd}\beta^{d-1}$. The previous computation and Lemma \ref{I condition} provide the cardinality estimate.
\end{proof}

\begin{lem}\label{size} If $M$ is $\beta$-balanced, $C_{max}(M) \in [2^i,2^{i+1}]$ and $\bar{v}\in \Delta_{1,2},$ then 
$M\bar{v} \in \Delta_{2^i\beta^{-1},2^{i+2}}$. 
\end{lem}
\begin{proof} This follows from the fact that $2^i \beta^{-1} \leq C_{j}(M) \leq 2^{i+1}$ for all $j$
 and $|M\bar{v}|=\underset{j=1}{\overset{d}{\sum}} C_j(M) \bar{v}[j].$
\end{proof}

\begin{proof}[Proof of Proposition \ref{hits well}] 
Recall that $S$ is a positive measure subset of $\Delta_{1,2}$. 
Consider $\mathcal{V}_k(U)$ given by Lemma \ref{matrices} and $H=\underset{M \in \mathcal{V}_k(U)}{\cup}MS$. 
By Lemma \ref{size} and Conclusion 2 of Lemma \ref{matrices} we have $H \subset p_{\Delta}^{-1}(U)\cap \Delta_{2^k\beta^{-1},2^{k+1}}$.
 By the first remark in this section $H \subset \underset{n=1}{\overset{\infty}{\cup}}\mathcal{I}^{-n}(S)$. 
 The union that defines $H$ is a disjoint union by Conclusion 3 of Lemma \ref{matrices} so 
 $\lambda(H) =\underset{M \in \mathcal{V}_k(U)}{\sum} \lambda(MS)$.
  Because our matrices have determinant 1 and so preserve $\lambda$ we have $\lambda(H)=\lambda(S)|\mathcal{V}_k(U)|$.
   By the cardinality estimate in Lemma \ref{matrices} this is at least $\lambda(S)c'_{\mathcal{I}}\lambda_0(U)2^{kd}$ and the proposition follows. 
\end{proof}

\begin{proof}[Proof of Corollary \ref{hits well perm}]
Given $U\subset \Delta$ open, by  Proposition \ref{hits well} there exists a permutation $\tau$ such that 
 $$\lambda(\underset{i=1}{\overset{\infty}{\cup}}\mathcal{I}^{-i}S \cap \Delta_{\beta^{-1}2^k,2^{k+1}} \cap p^{-1}_{\Delta}(U)\times\{\tau\})>\frac{1}{|\mathcal{P}|} C_{\mathcal{I}}\lambda_0(U)\lambda(\Delta_{\beta^{-1}2^k,2^{k+1}})\frac{\lambda(S)}{\lambda(\Delta_{1,2})}.$$

Let $\sigma \in \mathcal{P}$, then there exists $n$ and a matrix of induction $M((\bar{x},\sigma),n)$ such that $\pi_2(I^{n}(\bar{x}_{\tau},\sigma))=\tau$,
  and $(M(\bar{x},\sigma),n)_{\Delta} \subset U$. 
Considering the image of $$\underset{i=1}{\overset{\infty}{\cup}}\mathcal{I}^{-i}S \cap \Delta_{\beta^{-1}2^k,2^{k+1}} \cap p^{-1}_{\Delta}(U)\times\{\tau\}$$ under this matrix, the corollary follows with $q=2^n$.
\end{proof}
Corollary \ref{hits real well} follows similarly by using the above argument. Recall $S=S'\times \mathcal{P}$.
\begin{proof}[Proof of Corollary \ref{hits real well}] For each large $k$ by Lemma \ref{I condition}  
we have that 
$$\lambda(\underset{i=1}{\overset{\infty}{\cup}}\mathcal{I}^{-i}(S' \times \{\sigma'\}) \cap \Delta_{2^k\beta^{-1},2^{k+1}} \times \mathcal{P})\geq c_{\mathcal{I}}2^{kd}\lambda(S).$$ 
So there exists $\tau$ such that 
$$\lambda(\underset{i=1}{\overset{\infty}{\cup}}\mathcal{I}^{-i}(S'\times \{\sigma'\}) \cap \Delta_{2^k\beta^{-1},2^{k+1}} \times \{\tau\}) \geq
\frac{1}{|\mathcal{P}|} c_{\mathcal{I}}2^{kd}\lambda(S).$$ 
 For each $\sigma, \tau\in \mathcal{P}$ there exist $t, n_{\tau}$, and matrices of induction 
 $M_1$,...,$M_{t}$ such that $(M_i)_{\Delta}\subset U$, each 
 $M_i=M((\bar{x}_i,\sigma),n_i))$ where $n_i<n_{\tau}$, $\pi_2(I^{n_i}(\bar{x}_i,n_i))=\tau$,
   and $\lambda_0(\cup (M_i)_{\Delta})\geq \frac 1 2 \lambda_0(U)$.  
   Applying these matrices to 
   $\underset{i=1}{\overset{\infty}{\cup}}\mathcal{I}^{-i}(S'\times \{\sigma'\}) \cap \Delta_{2^k\beta^{-1},2^{k+1}} \times \{\tau\}$
   and letting $q'=2^{\max \, n_{\tau}}$ the corollary follows.
\end{proof}

\subsection{Proof of Proposition \ref{compare}}

Here is where Assumption 1 is used. We assume that there exists
$\tilde{N}$ which is a loop permutation matrix.

\begin{lem}\label{start} For any $\epsilon$ and $k$ there exists
$R_{\epsilon,k}:=R$ such that for all $(\bar y,\sigma) \in \Delta\times \mathcal{P}$ except a
set of $\lambda_0$-measure $\epsilon$ we have $M((\bar y,\sigma),1)M(I(\bar y,\sigma),1)\cdots M(I^R(\bar y,\sigma),1)$ contains a block of the form
$\tilde{N}^k$.
\end{lem}
\begin{proof} This follows because
$\lambda_0((\tilde{N}^m)_{\Delta})>0$ and the action of $I$ on $\Delta\times \mathcal{P}$ is $\lambda_0$-ergodic where $\lambda_0$ is a finite measure.
\end{proof}

Let $E_{i,j}$ denote the matrix with a $1$ in the $i^{th}$ row $j^{th}$ column and zeros elsewhere. Given a matrix $A$ let $A[i,j]$ denote the entry in the $i^{th}$ row $j^{th}$ column. Let $\tilde{N}$ be the matrix coming from taking the loop permutation to itself.
Consider $A\tilde{N}^kB=A(Id+kE_{i,j})B $ for some $i \neq j$ by our
assumption on $\tilde{N}$. This in turn is
$\underset{t=1}{\overset{d}{\sum}}\underset{s=1}{\overset{d}{\sum}}A[t,i]B[j,s]kE_{t,s}+AB$.

The idea of the proof of Proposition \ref{compare}  is that a fixed proportion
of the matrices $M\in \mathcal{V}_k(U)$ have the form $A\tilde{N}^nB$ where 
$$|A\tilde{N}^{n+r}B\bar{u}|-|A\tilde{N}^nB\bar{u}|\in [ra\epsilon2^k,r\epsilon 2^k]$$
 for any $\bar{u}\in \Delta$ and the $a$ depends only on $\mathcal{I}$. So by varying $r$
  we can move the measure in $\underset{i=1}{\overset{\infty}{\cup}}\mathcal{I}^{-i}(S)$ from
  $\Delta_{2^k(1+l\epsilon),2^{k}(1+(l+1)\epsilon)}$ to $\Delta_{2^k(1+r\epsilon),2^k(1+(r+1)\epsilon)}$.

\begin{proof}[Proof of Proposition \ref{compare}]

Observe that if $A\in \mathcal{M}_{\zeta}$ and 
$\frac{B[r,s]}{B[g,h]}\leq \kappa$ then for any $\bar{u}, \bar{v} \in \Delta$ we have
 $\frac{|AE_{i,j}B\bar{u}|}{|AE_{i,j}B\bar{v}|}\leq \kappa\zeta$ and $\frac{|AB\bar{u}|}{|AE_{i,j}B\bar{u}|}\leq d\kappa.$
 It follows that if $A$ and $B$ have these properties and $A\tilde{N}^mB\in \mathcal{V}_k(U)\subset \mathcal{M}_{\beta}$, then $|AE_{i,j}B \bar{x}|\in [\frac{2^k}{\beta \kappa d t},\frac{2^{k+1}}{t}] $, for every  $ \bar{x} \in \Delta$.

We claim that there exist $\zeta$ and $\kappa$ such that given $t$ for any large enough $k$ the proportion of matrices in $\mathcal{V}_k(U)$ with this form is at least a fixed constant independent of $t$. The condition on $B$ depends on $B_{\Delta}$ being away from the boundary of the simplex and $B$ being balanced. 
Given large enough $\zeta$, there are matrices $A'$, $B'$ such that $\lambda_0\left(\left(A'\tilde{N}^mB'\right)_{\Delta}\right)>0$, $A=\tilde{A}A'$ is $\zeta$-balanced and that $B=B' \tilde{B}$ has that $B_{\Delta}$ is away from the boundary of the simplex for any matrices of induction $\tilde{A}, \tilde{B}$.  Analogously to Lemma \ref{start} we have $A'\tilde{N}^mB'$ appears in most large enough matrices of induction. By Lemma \ref{I condition} at least $c_{\mathcal{I}}$ of these matrices have that $B$ is $\beta$-balanced.

Let 
\begin{multline*}
G_k(U)=\{A\tilde{N}^mB \in V_k(U): |AE_{i,j}B\bar{x}|\in [\epsilon \frac{2^{k-2}}{d^2 \zeta\beta\kappa},
 \epsilon \frac{2^{k-1}}{ \beta}] \forall \bar{x} \in \Delta\}.
 \end{multline*}
 It follows that there exist $\zeta,\kappa,\omega$ independent of $\epsilon$ such that for all large enough $k$ (the largeness of $k$ does depend on $\epsilon$)
$|G_k(U)| \geq \omega |\mathcal{V}_k(U)|$.
Let 
$$
G_k(U)[s]=\{A\tilde{N}^{m+s}B: A\tilde{N}^mB \in G_k(U)\}.
$$
For each $\bar{u} \in \Delta_{1,2}$ and $A\tilde{N}^mB\in G_k(U)$ there exists $s$ so that $A\tilde{N}^{m+s}B\bar{u} \in \Delta_{2^k(1+(l+1)\epsilon),2^k(1+(l+2)\epsilon)}$.
 Because our matrices all preserve $\lambda$, and for each fixed $s$ we have that the matrices in  $G_k(U)[s]$ have disjoint images (by Conclusion 3 of Lemma \ref{matrices}) we have 
 \begin{multline}
 \underset{s=1}{\overset{2d^2\beta\zeta\kappa}{\sum}}\lambda(\underset{M \in G_k(U)[s]}{\cup} MS \cap \Delta_{2^k(1+(l+1)\epsilon,2^k(1+(l+2)\epsilon)})\geq\\
  \lambda(\underset{M \in G_k(U)}{\cup}MS \cap \Delta_{2^k(1+l\epsilon),2^k(1+(l+1)\epsilon)}).
 \end{multline}
 Therefore there exists $1\leq s\leq d^2 \beta \zeta \kappa$ such that  if 
  \begin{multline}
 \lambda(\underset{M \in G_k(U)}{\cup}MS \cap \Delta_{2^k(1+l\epsilon),2^k(1+(l+1)\epsilon)}) \geq \\ 
 \hat{C} \lambda((p_{\Delta}^{-1}(U)\cap \Delta_{2^k(1+l\epsilon),
2^{k}(1+(l+1)\epsilon)})\times \{\sigma\}  \cap \underset{i=1}{\overset{\infty}{\cup}}\mathcal{I}^{-i}S)),
 \end{multline}
then
 \begin{multline}
 \lambda(
 \underset{M \in G_k(U)[s]}{\cup} MS 
 \cap (p_{\Delta}^{-1}(U) \cap \Delta_{2^k(1+(l+1)\epsilon),
2^{k}(1+(l+2)\epsilon)})\times \{\sigma\}  )\geq \\ \frac{\hat{C}}{2d^2\beta\zeta\kappa}\lambda((p_{\Delta}^{-1}(U) \cap \Delta_{2^k(1+l\epsilon),
2^{k}(1+(l+1)\epsilon)})\times \{\sigma\}  \cap \underset{i=1}{\overset{\infty}{\cup}}\mathcal{I}^{-i}S)).
\end{multline}
The proposition follows similarly in the general case.
\end{proof}

\bigskip

\subsection{Proof to Theorem \ref{main}}

We first prove a technical lemma whose proof is similar   to the proof of Proposition \ref{hits well} with some additional complications. 

\begin{lem}\label{run condition} Let $W$ be a  set that is $\mathcal{I}$ invariant with
$\lambda(W)>0$ and  let $\epsilon>0$. Then there exist $r_{\epsilon}>0$  
and  infinitely many arbitrarily small $a>0$
(depending on $W$), $\bar x_a\in \Delta_{a,2a}$ and $\sigma_a\in \mathcal{P}$ such that 
$$
\lambda(W \cap (B(\bar{x}_a,r_{\epsilon}a)  \times\{\sigma_a\}))\geq (1-\epsilon) \lambda(B(\bar{x}_a,r_{\epsilon}a)).
$$ 
\end{lem}

\begin{proof}
This follows from Lemma \ref{matrices} and the Lebesgue Density Theorem.  Let $\bar{x}$ be a density point of $W$, so that 
$\lambda (B(\bar{x},r )\cap W)> (1-\frac{\epsilon}{T}) \lambda(B(\bar{x},r))$ where $T$ depends only on $\mathcal{I}$, and the dimension $d$ and is specified in the last line of the proof. 
For simplicity we write $\bar{x}$ instead of $(\bar{x},\sigma) $.
Let $U=p_{\Delta}(B(\bar{x},\frac{r}2))$.

Let $\mathcal{V}_k(U)$ be the set of matrices guaranteed by Lemma \ref{matrices}. 
Let $\{\bar{z}_j\}_{j=1}^t$ be a maximal collection of $\frac{2r}{|\bar{x}|}$ separated points in $\Delta$. 
 For each $M \in \mathcal{V}_k(U)$ and $1\leq j \leq t$, let 
$\bar{w}_{M,j}= \frac{| \bar{x}|}{|M \bar{z}_j|}\bar{z}_j$, thus $M\bar{w}_{M,j}\in B(\bar{x},\frac r 2).$ Let $||M||_{op}$ denote the $L^1$ operator norm of $M$. This is $|C_{\max}(M)|$.
 We have 
$$MB(\bar{w}_{M,j},\frac{r}{2||M||_{op}})\subset B(\bar{x},\frac{r}2 +||M||_{op}\frac{r}{2||M||_{op}})=B(\bar{x},r).$$
 We will show that one of the $B(\bar{w}_{M,j},\frac{r}{2||M||_{op}})$ satisfies the conclusion of Lemma \ref{run condition} by showing that 
\begin{equation}\label{measure estimate}
\lambda(\underset{M \in \mathcal{V}_k(U)}{\cup}\underset{j=1}{\overset{t}{\cup}} MB(\bar{w}_{M,j},\frac{r}{2||M||_{op}}))>\frac 2 T \lambda(B(\bar{x},r)).
\end{equation} 
This is enough because then we will have 
$$
\lambda(\underset{M \in \mathcal{V}_k(U)}{\cup}\underset{j=1}{\overset{t}{\cup}} MB(\bar{w}_{M,j},\frac{r}{2||M||_{op}}))\geq 2\epsilon^{-1} \lambda(B(\bar x,r) \cap W^c)).
$$
The pigeonhole principle will then imply that there exists $MB(\bar{w}_{M,j},\frac{r}{2||M||_{op}})$ with 
$$
\lambda(MB(\bar{w}_{M,j},\frac{r}{2||M||_{op}})\cap W)\geq (1-\epsilon)\lambda(B(\bar{w}_{M,j},\frac{r}{2||M||_{op}})).
$$
 Because $W$ is $\mathcal{I}$ invariant and $M$ preserves volume it will follow that 
 $\lambda(B(\bar{w}_{M,j},\frac{r}{2||M||_{op}})\cap W)\geq (1-\epsilon)\lambda(B(\bar{w}_{M,j},\frac{r}{2||M||_{op}}))$ completing the lemma.
 
 To prove that we first show that $\underset{M \in \mathcal{V}_k(U)}{\cup}\underset{j=1}{\overset{t}{\cup}} MB(\bar{w}_{M,j},\frac{r}{2||M||_{op}})$ is a disjoint union.
  If $M\neq N$ then
  $MB(\bar{w}_{M,i},\frac{r}{2||M||_{op}}) \cap NB(\bar{w}_{N,j},\frac{r}{2||N||_{op}}) =\emptyset$ by Conclusion 3 of Lemma \ref{matrices}.
   Let $M=N$, $i \neq j$ and $|\bar{w}_{M,i}|\leq |\bar{w}_{M,j}|$.
    Notice that $\bar{w}_{M,i}$ and $\bar{w}_{M,j}\frac{|\bar{w}_{M,i}|}{|\bar{w}_{M,j}|}$
     are $|\bar{w}_{M,i}|\frac{2r}{|\bar{x}|}\geq \frac{r}{||M||_{op}}$ 
     separated. 
 So
   \begin{multline}\label{disjoint}
  p_{\Delta}(B(\bar{w}_{M,i},\frac{r}{2||M||_{op}})) \cap p_{\Delta}(B(\bar{w}_{M,j},\frac{r}{2||M||_{op}}))\subset \\   p_{\Delta}(B(\bar{w}_{M,i},\frac{r}{2||M||_{op}})) \cap p_{\Delta}(B(\bar{w}_{M,j}\frac{|\bar{w}_{M,i}|}{|\bar{w}_{M,j}|},\frac{r}{2||M||_{op}})) =\emptyset, 
   \end{multline}
as the points $\bar{z}_i, \bar{z}_j$ are separated.  Because $M$ is an injective map, by (\ref{disjoint}) 
 $$
 MB(\bar{w}_{M,i},\frac{r}{2||M||_{op}}) \cap MB(\bar{w}_{M,j},\frac{r}{2||M||_{op}}) =\emptyset
 $$ 
  and so the union is disjoint.
  
  Because $M$ preserves $\lambda$ it follows that
  $$
  \lambda(\underset{M \in \mathcal{V}_k(U)}{\cup}\underset{j=1}{\overset{t}{\cup}} MB(\bar{w}_{M,j},\frac{r}{2||M||_{op}}))=
  \underset{M \in \mathcal{V}_k(U)}{\sum} t \lambda(B(\bar{w}_{M,1},\frac{r}{2||M||_{op}}) ).
  $$ 
  By the cardinality estimate in Lemma \ref{matrices} this is at least 
  $$c'_{\mathcal{I}}\lambda_0(U)2^{kd}t\lambda(B(\bar{w}_{M,j},\frac{r}{2||M||_{op}}))=c_{\mathcal{I}}\lambda_0(U)2^{kd}te_1(\frac{r}{2\cdot2^{k+1}})^d$$ where $e_1$ depends only on the dimension. By the definition of $U$ and $t$, $\lambda_0(U)$ is at least $e_2(r/|\bar{x}|)^{d-1}$ and $t$ is at least  $e_3(r/|\bar{x}|)^{-(d-1)}$ where $e_2$ and $e_3$ depend only on the dimension. This is because $t$ is the cardinality of a maximal set of $c\frac{r}{|\bar {x}|}$ separated points in a simplex of dimension $d-1$ and $\lambda_0(U)$ is the measure of the projection of a sphere of radius $\frac 1 2 \frac{r}{|\bar{x}|}$ onto a $(d-1)$-dimensional hyperplane. So 
  $$\lambda(\underset{M \in \mathcal{V}_k(U)}{\cup}\underset{j=1}{\overset{t}{\cup}} MB(\bar{w}_{M,j},\frac{r}{2||M||_{op}}))>c'_{\mathcal{I}}e_1e_2e_32^{-2d}r^d>c'_{\mathcal{I}}e_1e_2e_32^{-2d}e_4\lambda(B(\bar{x},r))$$ where $e_4$ depends only on the dimension, establishing Equation (\ref{measure estimate}).  The lemma is completed and $a\in [|\bar{x}|2^{-k-1},|\bar{x}|\beta2^{-k}]$ and $r_{\epsilon}=\frac{r}{4\beta}$. $T$ can be chosen to be $(c'_{\mathcal{I}}e_1e_2e_3e_4)^{-1}2^{2d}(4\beta)^d$.
  \end{proof}

\begin{proof}[Proof of Theorem \ref{main}] Assume that $\mathcal{I}$ is not ergodic with respect to the measure $\lambda$. Therefore there
exists a measurable set $V \subset \mathbb{R}^d_+ \times \mathcal{P}$ such that
 \begin{enumerate}
 \item $\lambda(V)>0$
 \item $\lambda(V^c)>0$
 \item $\mathcal{I}(V) \subset V$
 \item $\mathcal{I}^{-1}(V) \subset V$.
\end{enumerate}
We will show that for $\lambda$-almost every $\bar z \times \{\sigma\}$ we have
 $$
 \underset{r \to 0^+}{\liminf} \frac{\lambda (B(\bar z,r) \times \{\sigma\}\cap V)}{\lambda(B(\bar z,r))}>0.
$$ 
 By Lemma \ref{run condition} there exists $r>0$ such that for arbitrarily small $a$ we have $B(\bar{x}_a,ar)\times \{\sigma_a\}$ with 
 $\lambda\left(V \cap B(\bar{x}_a,ra)\times \{\sigma_a\}\right)>(1-\epsilon)\lambda(B(\bar{x}_a,ra))$.
 By assumption $\lambda(V^c)>0$ and so $V^c$ has a Lebesgue density point $(\bar{z},\tau) \in \Delta_{b,2b}\times \mathcal{P}$, where 
 $b= \frac 3 2 |\bar {z}|$.

 By the $\lambda_0$-ergodicity of $I$ we may assume that there exists $n$ such that
     
 \begin{equation}\label{point condition}
     \underset{i=1}{\overset{n}{\cup}}\{\pi_2(I^i(\bar{z},\tau))\}=\mathcal{P} \text { and diameter}(M((\bar{z},\tau),m)) \to 0, \text { as } m \to \infty.
     \end{equation}
      Let $B(\bar{z},sb)\subset \Delta_{b,2b}$ such that
      $B(\bar{z},sb) \subset M((\bar{z},\tau),n) \mathbb{R}_+^d$ and
 $\lambda\left(B(\bar{z},sb)\times \{\tau\}\cap V^c\right)>(1-\epsilon)\lambda(B(\bar{z},sb))$.
 Let $U= p_{\Delta}(B(\bar{z},\frac{sb}2))$.
 Choose $k_0$ depending on $U$ so that Lemma \ref{matrices} and Proposition \ref{compare} hold. 
 Choose $a<\frac{b}{2^{k_0}}$, $B(\bar{x}_a,ar)$ and $\sigma_a$ 
 satisfying the conclusion of Lemma \ref{run condition}. 
 By Lemma \ref{matrices}  it follows that there is a set $\mathcal{V}_{\log(\frac{b}{a})}(U)$.
 Following the proof of Proposition \ref{hits well}, 
 
 \begin{multline}\label{point condition2}  
 \lambda(V \cap p_{\Delta}^{-1}(U)  \cap \Delta_{b\beta^{-1},2qb}) \leq 
  \lambda\left(\underset{i=0}{\overset{\infty}{\cup}}\mathcal{I}^{-i} V \cap  \left(V \cap p_{\Delta}^{-1}(U)
  \cap \Delta_{b\beta^{-1},2qb}\right)\right)\\ \geq
 \lambda \left( \underset{i=0}{\overset{\infty}{\cup}}\mathcal{I}^{-i}( B(\bar{x}_a,ra)\cap V) \cap  \left(V \cap p_{\Delta}^{-1}(U)
  \cap \Delta_{b\beta^{-1},2qb}\right)\right)\\ \geq
  \lambda \left( \underset{M \in \mathcal{V}_{\log(\frac b a )}(U)}{\cup} M(B(\bar{x}_a,ra)\cap V) \cap  \left(V \cap p_{\Delta}^{-1}(U)
  \cap \Delta_{b\beta^{-1},2qb}\right)\right)\\ =
  \lambda \left( \underset{M \in \mathcal{V}_{\log(\frac b a )}(U)}{\cup} M(B(\bar{x}_a,ra)\cap V) \right)
  \\ = |\mathcal{V}_{\log( \frac b a )}(U)|\lambda(B(\bar{x}_a,ra) \cap V).
 \end{multline}
 The second to last equality follows by Lemma \ref{size} and conditions 1 and 2 in the conclusion of Lemma \ref{matrices}. The last equality follows because  our matrices preserve $\lambda$ and because condition 3 in the conclusion of Lemma \ref{matrices} says the union is a disjoint union. By the cardinality estimate in Lemma \ref{matrices} and (\ref{point condition2}), we obtain
\begin{equation}\label{point condition3}
 \lambda(V \cap p_{\Delta}^{-1}(U)
  \cap \Delta_{b\beta^{-1},2qb})\geq C_1 r_{1/2 }^{d}\lambda(\Delta_{b,2b})\lambda_0(U).
 \end{equation}
  Letting $\epsilon=\frac{s}{3}$ so that there exists $l$ with 
  $$ p_{\Delta}^{-1}(U)\cap \Delta_{b(1+l\epsilon),b(1+(l+1)\epsilon)}\subset B(\bar{z},sb).
  $$
   Proposition \ref{compare} and (\ref{point condition3}) imply that 
$$ 
\lambda \left(V \cap p_{\Delta}^{-1}(U) \cap B\left(\bar{z},sb \right)\right)\geq
  C_2 r_{1/2 }^{d}\lambda\left(B\left(\bar{z},sb\right)\right).
$$

  We have that $(\bar{z},\tau)$ is not a Lebesgue density point for $V^c$ because $r_{1/2}>0$ is fixed and $q=2^n$ are fixed by $\bar{z},\tau$ (and not $s$). It follows that there is a $\lambda$-full measure set (the points satisfying  (\ref{point condition})  whose elements can not be  Lebesgue density points of $V^c$, therefore $V^c$ is a null measure set.
\end{proof}

\section{Selmer algorithm}

The aim of this section is to prove Corollary \ref{exact2}. In order to do it the {\it Selmer algorithm} \cite[p. 53-61]{sc}  will be related to a  variation of the Rauzy type algorithm so that Theorem \ref{main} can be applied. However, as we will see, Selmer algorithm does not satisfy  Assumption $1$ in Subsection $3.3$, 
therefore it is not a Rauzy induction type algorithm. Nevertheless we will prove that Selmer algorithm satisfies Assumption 3 in Subsection $4.1$ which is the statement of  \cite[Corollary 1.7]{ker}.

Let $\bar x=(x[1] ,\ldots, x[d]) \in \Sigma_d$ and $\sigma_{\bar x}$ be the permutation which arranges $x[1] ,\ldots, x[d-1],x[d]-x[1]$ in ascending order. If the point  $\bar x$ is irrational, then the iterates of the  map $S$ (\ref{selmer})
at $\bar x$  are  well defined.  
In matricial form the map $S$  equals
$$ 
S(\bar x)= \sigma_{\bar x} \left( \begin{array}{ccccccc}
1 &   0 &                & 0  & 0\\
0  & 1  &                & 0 & 0\\
   &      &  \ddots  && \\
0 & 0   &              &1&0 \\
-1 & 0   &           & 0 &1
\end{array} \right)
\left( \begin{array}{c}
x[1] \\
\vdots \\
\vdots \\
x[d]
\end{array} \right),
$$
therefore there is a nonnegative  elementary  matrix,  $M(\bar x,1)$, such that
$$
M(\bar x,1)S(\bar x)= \bar x.
$$

It is proved in  \cite{sc} that the cone 
$$
\Omega_d=\{\bar x \in \mathbb{R}_+^d: x[1] \leq  x[2] \leq \ldots \leq x[d] \leq x[1] + x[2] \}
$$ 
is an absorbing set for $S$ which means that
$$
\cup_{k\geq 0} S^{-k}( \Omega_d)= \mathbb{R}_+^d \;\; \mbox{ almost surely.}
$$
Let $\bar x\in \Omega_d$ be irrational, then there are two possibilities
$$
\sigma_{\bar x}=(1d2\cdots (d-1)) \; \mbox{ or } \sigma_{\bar x}=(d12\cdots (d-1)).
$$
We want to keep track how the coordinates have been exchanged to maintain the ascending order. So, if after $k$ iterations the permutation is $\sigma=(1'\cdots d')$, then after $k+1$ iterations the  permutations are
$$
\pi^{(1)}=(1'd'2'\cdots (d-1)') \; \mbox{ or } \pi^{(2)}=(d'1'2'\cdots (d-1)').
$$
The cone $\Omega$ is partitioned into two sub-cones
$$
\Omega^{(1)}= \{\bar x \in \Omega: x[d] \geq 2x[1]\} \; \mbox{ and }\;  
\Omega^{(2)}= \{\bar x \in \Omega: 2x[1] > x[d']\}
$$

Next throughout the section $S$ denotes the restriction of  (\ref{selmer})  to $\Omega$ 
$$
S:\bar x \in \Omega \longmapsto
\left\{ \begin{array}{ll}         
          (x[1],x[d]-x[1],x[2], \ldots ,x[d-1]), & \mbox{if $\bar x \in \Omega^{(1)}$} \\
          (x[d]-x[1], x[1],x[2], \ldots ,x[d-1]), & \mbox{if $\bar x \in \Omega^{(2)}$}
\end{array}
\right.
$$
and $S(\Omega^{(i)}=\Omega$, for  $i=1,2$.
In matricial form, if $\bar x \in \Omega$, the map $S$ writes
\begin{equation}\label{selmer matrix}
S(\bar x)= 
\left\{ \begin{array}{ll}
\left( \begin{array}{ccccccc}
1 &   0 &                & 0  & 0\\
-1  &  0 &                & 0 & 1\\
0  & 1  &                & 0 & 0\\
   &      &  \ddots  && \\
0 & 0   &           & 1 &0
\end{array} \right)
\left( \begin{array}{c}
x[1] \\
\vdots \\
\vdots \\
x[d]
\end{array} \right),\;\; \mbox{ if } \;\; \bar x \in \Omega^{(1)}(\pi)\\
\left( \begin{array}{ccccccc}
-1  &  0 &                & 0 & 1\\
1 &   0 &                & 0  & 0\\
0  & 1  &                & 0 & 0\\
   &      &  \ddots  && \\
0 & 0   &           & 1 &0
\end{array} \right)
\left( \begin{array}{c}
x[1] \\
\vdots \\
\vdots \\
x[d]
\end{array} \right),\;\; \mbox{ if } \;\; \bar x \in \Omega^{(2)}(\pi)).
\end{array}
\right.
\end{equation}

Let $\mathcal{P}$ be the set of all permutations on the $d$ letters $1,2, \ldots,d$. The dynamics of  $S$ defines an oriented graph $\mathcal{G}$ whose nodes are the permutations $\pi=(1'\cdots d') \in\mathcal{P}$. We find in Figure 1 a partial picture of the graph which indicates that every node $\pi=(1'\cdots d')$ has exactly two followers and two predecessors in the graph.

\begin{picture}(100,130)(-70,-20)

\put(40,80){\vector(1,-1){20}}

\put(60,45){$(1'\cdots d')$}

\put(40,20){\vector(1,1){20}}
\put(105,58){\vector(1,1){20}}

\put(0,85){$(2'3'\cdots d'1')$}
\put(-5,5){$(1'3'4'\cdots d'2')$}
\put(120,85){$(1'd'2'\cdots (d-1)')$}
\put(120,5){$(d'1'2'\cdots (d-1)')$}
\put(105,40){\vector(1,-1){20}}
\put(58,-10){{\bf Figure 1}}
\end{picture}

In the case $d=3$, where
$$
\mathcal{P}=\{(123),(132),(312),(213),(321),(231)\},
$$
the graph can be fully described, see Figure 2.

\begin{picture}(100,130)(-70,-20)
\put(0,85){$(123)$}

\put(120,85){$(132)$}
\put(15,80){\vector(1,-1){20}}
\put(40,90){\vector(1,0){60}}
\put(100,85){\vector(-1,0){60}}

\put(25,45){$(312)$}
\put(95,45){$(321)$}
\put(60,48){\vector(1,0){25}}
\put(85,44){\vector(-1,0){25}}
\put(35,40){\vector(-1,-1){20}}
\put(105,58){\vector(1,1){20}}

\put(0,5){$(231)$}
\put(120,5){$(213)$}
\put(125,18){\vector(-1,1){20}}
\put(40,10){\vector(1,0){60}}
\put(100,5){\vector(-1,0){60}}
\put(135,80){\vector(0,-1){60}}
\put(10,20){\vector(0,1){60}}

\put(55,-12){{\bf Figure 2}}
\end{picture}

\subsection{Proof of Corollary \ref{exact2}}

Using the graph $\mathcal{G}$, we define  a map which is a  version of the map $S$ with the  setup of a variation of the Rauzy induction type algorithm 
$$
\mathcal{S}: \Omega \times \mathcal{P} \to  \Omega\times \mathcal{P}
$$
given by
$\displaystyle
\mathcal{S}(\bar x, \pi) =
\left\{ \begin{array}{ll}         
          (S(\bar x), \pi^{(1)}) & \mbox{if $\bar x \in \Omega^{(1)}$} \\
          (S(\bar x), \pi^{(2)}), & \mbox{if $\bar x \in \Omega^{(2)}$.}
\end{array}
\right.$

Next we show that the graph $\mathcal{G}$ is {\it path connected} which means that any two nodes are linked by a path. 

\begin{lem} 
The graph $\mathcal{G}$ is path connected.
\end{lem}
\begin{proof} It suffices to show that any permutation $\pi=(1'\cdots d')$ can be transformed into the permutation $(1\cdots d)$ through a sequence of shifts and $vice$ $versa$. 

The shift on the graph is defined by the following operations:
$$
A(1'\cdots d')=(1'd'2'\cdots (d-1)') \; \mbox{ and } B(1'\cdots d')=(d'1'2'\cdots (d-1)').
$$

Assume that the block $12\cdots i$ appears in $\pi$. If $i=d$, we are done.

Let $1\leq i <d$ and we have $\pi=(1' \cdots a' 12\cdots i (a+i+1)'\cdots d')$. 
Next we will shift $\pi$ in order to obtain a permutation which contains the block  $12\cdots i(i+1)$.
Let $m=d-a$. A suitable shifting of $\pi$ gives
$$
B^m(1' \cdots a' 12\cdots i (a+i+1)'\cdots d')= (12\cdots i (a+i+1)'\cdots d'1' \cdots a').
$$
So it suffices to consider $\pi=(12\cdots i (i+1)'\cdots d')$, where $(i+1)'\neq i$.
Therefore the element $i+1$ belongs to the block $(i+1)'\cdots d'$ which means that there exists $i+1\leq b \leq  d$ such that $b'=i+1$. Thus setting  $k=d-b+1$,  we obtain 
$$
B^k(\pi)=((i+1)(b+1)'\cdots d'12\cdots i (i+1)'\cdots (b-1)').
$$
Setting $\ell=b-i+1$, we obtain 
$$
A^{\ell}(B^k(\pi))=((i+1)(i+1)'\cdots d'  12\cdots i ),
$$
finally
$$
B^i((A^{\ell}(B^k(\pi)) )=(12\cdots i(i+1)(i+1)'\cdots d'  ).
$$

This proves the claim.
\end{proof}

Using the graph we define  a map which is a  version of the map $S$ with the  set up which is a variation of the Rauzy induction type algorithm 
$$
\displaystyle \mathcal{S}:\cup_{\pi \in \mathcal{P}} \Omega(\pi) \to \cup_{\pi \in \mathcal{P}} \Omega(\pi),
$$
given by, if
$ \bar x \in \Omega(\pi)^{(i)}$, then $\mathcal{S}( \bar x) \in \Omega(\pi_i)$, for $i=1,2$.

\begin{lem} 
Let $1\leq i \leq d$ and $\mathcal{P}'=\{ \pi=(1'\cdots d') \in \mathcal{P}: 1'\neq i\}$, then the set $\displaystyle \cup_{\pi \in \mathcal{P}'} \Omega(\pi)$ is a null Lebesgue measure set.
\end{lem}
\begin{proof}  
For simplicity let $i=d$. Let $\bar x \in \displaystyle \cup_{\pi \in \mathcal{P}'} \Omega(\pi)$ be irrational and $\bar x^{(k)}=\mathcal{S}^k( \bar x)$, for every $k\geq 0$. As $x[d]^{(k)}> y^{(k)}=\min_{1\leq j \leq d} x[j]^{(k)}$, for every $k\geq 0$, we have that 
$$
x[d]^{(k(d-1))}=x[d]-y^{(0)}-\ldots - y^{(k(d-1))}
$$
which implies that $x[d]= \sum_{k\geq 0} y^{(k(d-1))}$. On the other hand, the sequence $y^{(n)}$ depends only on $x[1],\ldots, x[d-1]$. Therefore
the set $\displaystyle \cup_{\pi \in \mathcal{P}'} \Omega(\pi)$ is a null Lebesgue measure set.
\end{proof}

On the other hand, the Selmer map $\mathcal{S}$ does not satisfy Proposition $1.4$ in \cite{ker}. It means in the language of \cite{ker} that there exist sets of isolated columns which are not the entire set of columns. Therefore the map $\mathcal{S}$ is not a Rauzy induction type algorithm, as it does not satisfy Assumption $1$ in Subsection $3.3$. Nevertheless next we will prove that the map $\mathcal{S}$ satisfies Assumption 3 in Subsection $4.1$ which is the statement of \cite[Corollary 1.7]{ker}.

\subsection{Proof of Assumption 3 Subsection 4.1}

We consider the Selmer Algorithm $S$ given by equation (\ref{selmer}) restricted to the absorbing set $\Omega=\Omega_d$.
Let $\bar{S}$ denote the normalized Selmer algorithm on $d$ letters. Let $\Omega_1=\Omega \cap \Delta$. Let $\lambda_0$ denote Lebesgue measure on $\Delta$. Let $p$ be the preserved measure for $\bar{S}$ on $\Omega_1$. Let the matrix $M$ denote a fixed past. So $M$ is given by the product of inverses of matrices as in equation (\ref{selmer matrix}). We study the measure $p(U|M)$, the conditional measure given a fixed past.
 If $v$ is a vector let $v_i$ denote its $i^{th}$ entry.  Critical position refers to the first position.

The aim of this section is to establish the next result which means that Assumption 3 in Subsection 4.1 also holds for Selmer algorithm.

\begin{prop} \label{selmer balanced general}There exists $\beta, K,\rho$ such that the conditional probability that the matrix will become $\beta$-balanced before the largest column increases by a factor of K is at least $\rho$ regardless of the past.
\end{prop}

We recall from Veech \cite[Proposition 5.2]{ietv} that the conditional probability, $p(U|M)$, given a fixed matrix from the past $M$ is 
\begin{equation}\label{jac estimate}
\int_U\frac{dv_1\cdots dv_d}{(C_1v_1+...+C_dv_d)^d}
\end{equation}
 where $C_i$ is the sum of the $i^{th}$ column of $M$.

\begin{prop} \label{v1 decay} 
For any $0<\eta<1$ and any $\epsilon>0$ there exists $L_{\eta,\epsilon}:=L$ such that with probability $\eta$ there exists $i$ such that $S^i(\bar{v})_1<\epsilon v[ 1]$ and the largest column hasn't increased by a factor of $L$.
\end{prop}

\begin{lem}\label{leave critical}
For any $0<\rho<1$ there exists $K_{\rho}$ so that for any past the conditional probability that the first entry leaves the critical position before the largest column increases by a factor of $K_{\rho}$ is at least $\rho$.
\end{lem}
Let $U_{r,M}=\{\bar{v}\in \Omega_1: v_1\leq r \frac{\max\{C_i(M):i>1\}}{C_1(M)}\}$. 

\begin{claim}\label{bounded jac}
There exists a constant $D$ such that for any $\bar{u},\bar{w}\in U_{1,m}$ we have
$$
\frac 1 D<\frac{C_1u_1+...+C_du_d}{C_1w_1+...+C_dw_d}<D .
$$
Thus the ratio of the Radon derivative of $p(\cdot |M)$ is bounded on $U_{1,M}$ independent of $M$.
\end{claim}
\begin{proof} 
If $\bar{u}\in U_{1,M}$ then 
$$\underset{i=1}{\overset{d}{\sum}} C_iu_i\leq \max\{C_i:i>1\}+\sum_{j=2}^du_i\max\{C_i:i>1\}< 2\max\{C_i:i>1\}.$$
 If $\bar{u}\in \Omega_1$ then $u_d\geq \frac 1 d$ and so $u_i\geq \frac 1 {2d}$ for all $i>1$. So 
$$ \sum_{i=1}^du_iC_i\geq \frac 1 {2d}\, \max\{C_i:i>1\}.$$
This establishes the claim with $D=4d$.
\end{proof}

\begin{claim} \label{small meas}
There exists a constant $r$ independent of $M$ such that $\frac{\lambda(U_{\epsilon,M})}{\lambda(U_{1,M})}<r\epsilon^{d-1}$.
\end{claim}
\begin{proof} The sets
$U_{t,M}$ are $(d-1)$-dimensional polygons that are determined by a vertex with sides emanating from it with lengths proportional to $t\frac{\max\{C_i:i>1\}}{C_1}$ and so that each side makes a definite angle with the span of the other sides. 
\end{proof}

\begin{proof}[Proof of Lemma \ref{leave critical}] 
Observe that $v_1$ leaves the critical position in at most $\frac{1}{v_1}+d$ steps. 
By Claim \ref{bounded jac}
$$ \frac{p(U_{t,M})}{p(U_{1,M})}\leq D^d \frac{\lambda(U_{t,M})}{\lambda(U_{1,M})}$$
 and so by Claim \ref{small meas} there exists $c_{\rho}$ such that 
$$p(U_{c_{\rho},M}|M)<(1-\rho)p(U_{1,M}|M).$$
Trivially, the right hand side is bounded by $\rho$. 
 So with conditional probability at least $\rho$ we have that $v_1\geq c_{\rho}\frac{\max\{C_i:i>1\}}{C_1}$ and $v_1$ leaves the critical position in at most $\frac{C_1}{\max\{C_i:i>1\}c_{\rho}}$ steps. The largest column increases by at most $\frac{C_1}{c_{\rho}}+3C_1$. 
\end{proof}

\begin{proof}[Proof of Proposition \ref{v1 decay}] 

\begin{claim} 
Under $S$, whenever a term enters the critical position it means that it has decreased by at least half since it was last in the critical position. 
\end{claim}
\begin{proof} 
Consider $(a,...,y)$ the step before it has left critical position. So  $S(a,...,y)=(y-a,a,...,z)$ and $S^{d-2}(y-a,a,...,z)=(w,...,a)$ for some $w$. If $w\geq \frac 1 2 a$ then $S(w,...,a)=(a-w,w,...,z')$ and $a-w\leq \frac 1 2 a$. If $w<\frac 1 2 a$ then $a$ must decrease to be smaller than $w$ before it enters critical position again. 
\end{proof}

 It follows that if the term in the critical position has changed at least $k$ times then one term has been in critical position at least $\frac k d $ times and the first time it is in critical position it is at most the initial minimum entry. When this entry is first in critical position it is at most the initial smallest entry. Thus the minimum entry has decreased by at least $2^{-\lfloor \frac k d \rfloor+1}$. By the Lemma \ref{leave critical} this happens with probability at least $\rho^k$ before the largest column increases by $2^kK_{\rho}^k$. For $\epsilon>2^{-\lfloor \frac k d\rfloor+1}$ and $\eta<1$ choose $\rho=\eta^{\frac 1 k}$ and $L=2^kK_{\rho}^k$.
\end{proof}

\begin{cor}\label{v12 decay}  
For any $\rho<1$, $\epsilon>0$ there exists $D_{\rho,\epsilon}:=D$ such that with probability $\rho$ there exists $i$ such that $S^i(\bar{v})_2<\epsilon v[1]$ and the largest column hasn't increased by a factor of $D$. 
\end{cor}

We now set up some notation: Let $\sigma_{\bar{v}}$ be the permutation on $d$-letters that puts 
$(v[ 1],...,v[ d-1],v[ d]-v[ 1])$ in ascending order. Let 
$\tau_{\bar{v},k}=\sigma_{S^k(\bar{v})}\circ...\circ \sigma_{\bar{v}}$. 

\begin{lem}\label{combine}
 If $S^i(\bar{v})_1+S^i(\bar{v})_2<v[ d]-v[ 1]-v[ 2]$ then $\tau_{\bar{v},j}^{-1}(1)=d $ for some $j\leq i$. Thus the entry corresponding to $v[ d]$ has been in the critical position.
\end{lem}
\begin{proof}
We first claim that if $\tau^{-1}_{\bar{v},j}(1)\neq d$ for $j=1$ and so every $j\leq d-1$ then 
$$
S^{d-1}(\bar{v})_1+S^{d-1}(\bar{v})_2-S^{d-1}(\bar{v})=v[ 1]+v[ 2]-v[ d].
$$
 This follows because under our assumption 
 $$
S^{d-1}(\bar{v})_1+S^{d-1}(\bar{v})_2=v[ 2] \text{ and } {S^{d-1}(\bar{v})_d=v[ d]-v[ 1]}.
$$ 
To see that $S^{d-1}(\bar{v})_1+S^{d-1}(\bar{v})_2=v[ 2]$ observe that 
 $$
(S^{d-1}(\bar{v})_1,S^{d-1}(\bar{v})_2)=(v[ 2]-S^{d-2}(\bar{v})_1,S^{d-2}(\bar{v})_1) \text{ or }(S^{d-2}(\bar{v})_1,v[ 2]-S^{d-2}(\bar{v})_1).
$$
By induction if  $\tau^{-1}_{\bar{v},j}(1)\neq d$ for $j\leq k(d-1)$ then 
$$
S^{d-1}(\bar{v})_1+S^{d-1}(\bar{v})_2-S^{d-1}(\bar{v})_d=v[ 1]+v[ 2]-v[ d].
$$
  So if $S^{k(d-1)}(\bar{v})_1+S^{k(d-1)}(\bar{v})_2< v[ 1]+v[ 2]-v[ d]$, then  $\tau^{-1}_{\bar{v},j}(1)= d$ for some ${j\leq k(d-1)}$.
 \end{proof}
\begin{lem}\label{jacobian bound} 
If $(u_0,u_2,u_3,...,u_d)$ and $(u_0,w_2,w_3,...,w_d) \in \Omega_1$ then $\frac{u_i}{w_i}\in [\frac 1 {4} ,4]$ for all $i\geq 2$.
\end{lem}
\begin{proof}
The largest $u_d,w_d$ can be is $\frac{2}{d+1}$ (if $u_i=\frac 1 {d+1}$ for all $i<d$). The smallest $u_i,w_i$ can be is $\frac 1 {2d}$ for $i>1$. So all entries are between $\frac{1}{2d}$ and $\frac{2}{d+1}$.
\end{proof}
 Let $x$ be fixed and  define
$$
\Lambda(x):=\{\bar{u}=(u[1],...,u[d ])\in \Omega_1:u[1 ]=x\}.
$$

Let $\lambda_{d-2}$ be the natural $(d-2)$-dimensional volume on each $\Lambda(x)$.

\begin{cor}\label{vd not big} 
There exists a constant $r>0$ such that for any $M$ 
$$
p(\{\bar{v}\in \Omega_1: v_d>v[1]+v[2]-rv[1]\}|M)<\frac 1 2.
$$ 
\end{cor}
\begin{proof} 
By Lemma \ref{jacobian bound}, Equation (\ref{jac estimate}) and Fubini's Theorem we have that if
$U\subset \Omega_1$ then $p(U|M)<4^d\, \underset{x}{\sup}\frac{\lambda_{d-2}(\Lambda(x)\cap U)}{\lambda_{d-2}(\Lambda(x))}$. The corollary follows.
\end{proof}

\begin{lem}\label{d critical} 
There exists a constant $\kappa$ so that the conditional probability, given any fixed past matrix $N$, that $\tau_{\bar{v},j}^{-1}(1)=d$ for some $j$ where $ C_{\max}(NM(\bar{v},j))<\kappa C_{max}(N)$ is at least a quarter independent of $N$.
\end{lem}
\begin{proof}[Proof of Lemma \ref{d critical}]  
By Corollary \ref{vd not big}
$$
{p(\{\bar{v}\in \Omega_1: v[d]<v[1]+v[2]-rv[1]}\}|M)\geq \frac 1 2.
$$ 
Now by Corollary \ref{v12 decay} with $\rho=\frac 3 4$ and $\epsilon=\frac r {2}$ there is at least a three quarters chance that the sum of the first and second entries is at most $\frac{r}{2} v[1]$ before the largest column has increased by a factor of $D_{\frac 3 4 , \frac{r}{2}}$. By Lemma \ref{combine} if  $S^i(\bar{v})_1+S^i(\bar{v})_2< v[1]+v[2]-v[d]$ then $\tau_{\bar{v},j}^{_1}(1)=d$ for some $j\leq i$.  Combining these facts, there is at least a $\frac 1 2 -\frac 1 4 =\frac 1 4$ chance that the entry corresponding to $\bar{v}_d$ has entered the critical position before the largest column has increased by a factor of $D_{\frac 3 4 , \frac r {2}}$.
\end{proof}

\begin{proof}[Proof of Proposition \ref{selmer balanced general}] This follows by iterating Lemma \ref{d critical}.
\end{proof}

This concludes the proof of Corollary \ref{exact2}.

\section{Others homogeneous algorithms}

Next we give examples  of  others $d$-dimensional   generalizations of (\ref{E1}). 

\subsection{Jacobi-Perron algorithm}

Nogueira \cite{nog}   introduces  a way to study the Jacobi-Perron algorithm using  a {\it tagged subtractive algorithm}.
For simplicity we define a subtractive algorithm  in the  sub-cone  
$$\Gamma_d=\{\bar x=(x[1]  ,\ldots, x[d] ) \in \mathbb{R}_+^d: x[d] = \max_i x[i]  \}.$$ For $\bar x \in \Gamma_d$, let $i_0=i(\bar x)=\min \{i \geq 2: x[i] >x[1] \}$
 and
$$
\mathcal{J}(\bar x)= 
\left\{ \begin{array}{ll}         
(x[1] ,\ldots, x[i_0] -x[1] ,\ldots, x[d] ), & \mbox{if $i_0 < d$,} \\
(x[1]  ,\ldots, x[d-1] , x[d] -x[1] ), & \mbox{if $i_0 =d,x[d] >2x[1] $,} \\
 (x[2]  ,\ldots, x[d-1] , x[d] -x[1] ,x[1] ), & \mbox{if $i_0 =d,x[d] <2x[1] $.}          
         \end{array}
\right.
$$
The {\it homogeneous Jacobi-Perron algorithm} (see \cite{sc},  p. 24) is the map $\Gamma_d \to \Gamma_d$ given by, for every $ \bar x \in \Gamma_d$,
$$
\bar x \mapsto\mathcal{I}^{k(\bar x)}(\bar x)=\left(x[2]- \left[ \frac{x[2]}{x[1]}\right] x[1],\ldots,  x[d]- \left[ \frac{x[d]}{x[1]}\right] x[1],x[1]\right),
$$
where $\displaystyle k(\bar x) =\sum_{i=2}^d \left[ \frac{x[i]}{x[1]}\right]$ and $[x]$ means the integer part of the real number $x$.

\medskip

\subsection{Multidimensional continued fractions with absorbing sets}

For completeness sake next we give examples of homogeneous multidimensional continued fraction algorithms which are not ergodic. 

The first map is the Poincar\'e multidimensional algorithm which was introduced by Poincar\'e \cite{Poin} 
$$
\mathcal{J}: \bar x \in \Lambda_d \mapsto \sigma_{\bar x}(x[1],x[2]-x[1],\ldots,x[d]-x[d-1])\in \Lambda_d,
$$
where $\Lambda_d=\{\bar x \in  \mathbb{R}_+^d: x[1] \leq \ldots \leq x[d]\}$ and $\sigma_{\bar x}$ is the permutation which arranges $x[1],x[2]-x[1],\ldots,x[d]-x[d-1]$ in ascending order. This map corresponds to  nice geometrical interpretation of the one-dimensional continued fraction which was extended to $\mathbb{R}_+^d$. The aim of Poincar\'e was to derive good diophantine approximation from his algorithm. In Nogueira \cite{poin} it is proved that already for $d=3$ for Lebesgue almost every $\bar x$ in $\Omega_3$, there exists a positive  function $s(\bar x)>0$ such that
$$
\lim_{k \rightarrow \infty} \mathcal{J}^k(\bar x)=(0,0, s(\bar x))
$$
which implies that the algorithm does not give diophantine approximation and is not ergodic.
There it is conjectured that $\mathcal{J}$ is ergodic if, and only if, $d$ is even.

The second example arises from a model in Percolation Theory and the multidimensional algorithm is defined by
$$
\mathcal{J}: \bar x \in \Lambda_d \mapsto \sigma_{\bar x}(x[1],x[2]-x[1],\ldots,x[d]-x[1])\in \Lambda_d.
$$
In \cite{km}, Kraaikamp and Mester prove that, for $d\geq3$, for Lebesgue almost every $\bar x\in\Lambda_d$, there exists a positive  function $p(\bar x)>0$ such that
$$
\lim_{k \rightarrow \infty} \mathcal{J}^k(\bar x)=(0,\ldots, 0, p(\bar x)),
$$
where $p(\bar x)$ is related to the critical probability of the percolation model, therefore the map is not ergodic.\\

\noindent
{\bf Acknowledgements.}\\
We would like to thank Tomasz Miernowski for fruitful discussions.
The first named author (J.C.) was supported by CODY and NSF grant DMS-1004372 while working on this project. He would also like to thank for hospitality the Institut de Math\'ematiques de Luminy, Marseilles, where this work was completed.

\end{document}